\documentclass{amsart}
\usepackage{hyperref}
\usepackage{xcolor}

\newcommand{\R}{\mathbf{R}}
\newcommand{\Q}{\mathbf{Q}}

\newcommand{\N}{\mathbf{N}}
\newcommand{\tto}{\Rightarrow}

\newcommand{\eps}{\epsilon}

\newcommand{\cF}{\mathcal{F}}

\newcommand{\cP}{\mathcal{P}}

\DeclareMathOperator{\E}{\mathbf{E}}
\let\P\relax
\DeclareMathOperator{\P}{\mathbf{P}}
\DeclareMathOperator{\I}{\mathbf{1}}

\DeclareMathOperator\rank{rank}
\DeclareMathOperator\tr{tr}

\theoremstyle{plain}
\newtheorem{theorem}{Theorem}%[section]

\newtheorem{lemma}[theorem]{Lemma}
\newtheorem{corollary}[theorem]{Corollary}

\theoremstyle{definition}

\newtheorem{remark}[theorem]{Remark}

\theoremstyle{remark}

\numberwithin{equation}{section}
\numberwithin{theorem}{section}

\newcommand{\musc}{\mu_{\mathrm{sc}}}
\newcommand{\w}{w_{ij}}

\newcommand{\fl}[1]{\lfloor{#1}\rfloor}
\newcommand{\ce}[1]{\lceil{#1}\rceil}
\newcommand{\bi}{\mathbf{i}}
\newcommand{\bj}{\mathbf{j}}
\newcommand{\bc}{\mathbf{c}}

\begin{document}
\title[Characterization of convergence to the semicircle]{Necessary and
Sufficient Conditions for Convergence to the Semicircle Distribution}
\author{Calvin Wooyoung Chin}
\date{\today}
\begin{abstract}
We consider random Hermitian matrices with independent upper triangular entries.
Wigner's semicircle law says that under certain additional assumptions,
the empirical spectral distribution converges to the semicircle distribution.
We characterize convergence to semicircle in terms of the variances of
the entries, under natural assumptions such as the Lindeberg condition.
The result extends to certain matrices with entries having infinite second moments.
As a corollary, another characterization of semicircle convergence is given
in terms of convergence in distribution of the row sums to the standard normal distribution.
\end{abstract}
\maketitle

\section{Introduction}
Let $W_n$, for each $n \in \N$, be a random $n\times n$ Hermitian
matrix whose upper triangular entries are independent.
We call $(W_n)_{n\in\N}$ a \emph{Hermitian Wigner ensemble}.
In case $W_n$ is real symmetric for all $n\in\N$, we call $(W_n)_{n\in\N}$
a \emph{symmetric Wigner ensemble}.
We write $W_n = (w_{ij})_{i,j=1}^n$ throughout this paper.
If $\lambda_1(W_n) \ge \cdots \ge \lambda_n(W_n)$ are the eigenvalues of $W_n$
counted with multiplicity, then the \emph{empirical spectral distribution}
$\mu_{W_n}$ of $W_n$ is defined by
\[ \mu_{W_n} := \frac{1}{n}\sum_{i=1}^n \delta_{\lambda_i(W_n)}. \]
Since $\mu_{W_n}$ is a random measure, we can think of the mean measure
$\E\mu_{W_n}$, which is defined and treated in Appendix \ref{app:exp_meas}.

Let us use the term \emph{semicircle law} to refer to a class of
theorems that state, under certain conditions, that $\mu_{W_n}$ converges
in some sense to the \emph{semicircle distribution} $\musc$ on $\R$ given by
\[ \musc(dx) := \sqrt{(4-x^2)_+}\,dx. \]
(We let $x_+ := \max\{x,0\}$.) Wigner initiated the spectral study of
random matrices by proving the following very first version of the
semicircle law in \cite{Wig55, Wig58}.

\begin{theorem}[semicircle law, Wigner]
Let $(W_n)_{n\in\N}$ be a symmetric Wigner ensemble such that
the upper triangular entries of $W_n$ have identical symmetric
distribution with mean zero and variance $1/n$.
If for each $k\in\N$ we have
\begin{equation} \label{eq:wigner_bound}
\E|\w|^k \le B_kn^{-k/2} \qquad \text{for some $B_k < \infty$
independent of $n,i,j$,}
\end{equation}
then $\E\mu_{W_n} \tto \musc$.
Here $\tto$ denotes convergence in distribution.
\end{theorem}

Subsequent works by \cite{Arn71}, \cite{Pas73}, and others led to
the following much more general semicircle law.

\begin{theorem}[semicircle law, {\cite[Theorem 2.9]{BS10}}] \label{thm:wigner_bs}
Let $(W_n)_{n\in\N}$ be a Hermitian Wigner ensemble such that
the upper triangular entries of $W_n$ are of mean zero and variance $1/n$.
If
\begin{equation} \label{eq:lind}
\frac{1}{n} \sum_{i,j=1}^n \E[|\w|^2;|\w|>\eps] \to 0
\qquad \text{for all $\eps > 0$,}
\end{equation}
then $\mu_{W_n} \tto \musc$ a.s.
\end{theorem}

Note that \eqref{eq:wigner_bound} for $k=3$ implies \eqref{eq:lind} since
\[ \frac{1}{n} \sum_{i,j=1}^n \E[|\w|^2;|\w|>\eps]
\le \frac{1}{\eps n} \sum_{i,j=1}^n \E|\w|^3
\le \frac{B_3}{\eps\sqrt{n}} \to 0. \]
Let us call \eqref{eq:lind} the \emph{Lindeberg condition},
following the Lindeberg--Feller central limit theorem.
Girko \cite[Theorem 9.4.1]{Gir90} states that the converse of
Theorem \ref{thm:wigner_bs} holds.

Rather surprisingly, we have the following:
\begin{lemma}[a.s.\ convergence] \label{lem:as_reduct}
Let $(W_n)_{n\in\N}$ be a Hermitian Wigner ensemble.
Then
\[ \E\mu_{W_n}\tto\musc \quad \text{if and only if} \quad
\mu_{W_n}\tto\musc \text{ a.s.} \]
\end{lemma}
A proof of this fact using a concentration-of-measure inequality
is given in Appendix \ref{app:exp_meas}.
Thanks to this equivalence, we will be able to go back and forth freely
between the two types of convergences throughout the paper.

Theorem \ref{thm:wigner_bs} suggests an extension of the semicircle law
to the case where the entries of $W_n$ have variances other than $1/n$.
Here is one possible approach to such an extension.
Assume that the underlying probability space is the product
\[ (\Omega_1 \times \Omega_2,\cF_1\times\cF_2,\P_1\times\P_2) \]
of two probability
spaces $(\Omega_1,\cF_1,\P_1)$ and $(\Omega_2,\cF_2,\P_2)$.
Then let $X_n = (x_{ij})_{i,j=1}^n$ and $Y_n = (y_{ij})_{i,j=1}^n$
be random real symmetric matrices defined on $\Omega_1$ and $\Omega_2$
having i.i.d.\ upper triangular entries.
If $x_{11}$ is standard normal and
\[ \P_2(y_{11} = 1) = \P_2(y_{11} = -1) = 1/2, \]
then it is not difficult to show that $(W_n)_{n\in\N}$ given by
\[ w_{ij}(\omega_1,\omega_2) := x_{ij}(\omega_1)y_{ij}(\omega_2)/\sqrt{n} \]
satisfies the conditions of Theorem~\ref{thm:wigner_bs}.

Since $\mu_{W_n(\omega_1,\omega_2)} \tto \musc$ for $\P$-a.e.\
$(\omega_1,\omega_2)$, Tonelli's theorem implies that for $\P_1$-a.e.\
$\omega_1 \in \Omega_1$, we have $\mu_{W_n(\omega_1,\cdot)} \tto \musc$
$\P_2$-a.s.
Note that the $(i,j)$-entry of the random matrix $W_n(\omega_1,\cdot)$
defined on $(\Omega_2,\cF_2,\cP_2)$ has variance $x_{ij}(\omega_1)^2/n,$
which can deviate by any amount from $1/n$.

A problem with this approach is that we do not know for which $\omega_1$
we have the a.s.\ convergence $\mu_{W_n(\omega_1,\cdot)} \tto \musc$,
even though we know this happens for almost all $\omega_1$.
For instance, the above discussion does not tell us whether
$\mu_{W_n} \tto \musc$ a.s.\ is true when $(W_n)_{n\in\N}$ is a symmetric
Wigner ensemble such that
\[ \P(\w = \sqrt{2/n}) = \P(\w = -\sqrt{2/n}) = 1/2
\qquad \text{if $i+j$ is even} \]
and $\w = 0$ if $i+j$ is odd.

G\"otze, Naumov, and Tikhomirov \cite{GNT15} covered this case
by proving the following:

\begin{theorem}[semicircle law, {\cite[Corollary 1]{GNT15}}]
\label{thm:wigner_gnt}
Let $(W_n)_{n\in\N}$ be a symmetric Wigner ensemble such that
$\E\w = 0$ and $\E|\w|^2 < \infty$ for $i,j=1,\ldots,n$.
If the Lindeberg condition \eqref{eq:lind} holds, and
\begin{equation} \label{eq:row_one}
\frac{1}{n} \sum_{i=1}^n \Bigl| \sum_{j=1}^n \E|\w|^2 - 1 \Bigr| \to 0,
\end{equation}
and
\begin{equation} \label{eq:row_bdd}
\sum_{j=1}^n \E|\w|^2 \le C \qquad \text{for some $C<\infty$
independent of $n$ and $i$,}
\end{equation}
then $\mu_{W_n} \tto \musc$ a.s.
\end{theorem}

From our main result (Theorem~\ref{thm:wigner_var}) it will follow that
\eqref{eq:row_bdd} is not needed in Theorem~\ref{thm:wigner_gnt},
and that $(W_n)_{n\in\N}$ can be assumed to be Hermitian,
not necessarily real symmetric.

To illustrate that \eqref{eq:row_one} is needed in Theorem \ref{thm:wigner_gnt},
the authors of \cite{GNT15} considered the random symmetric block matrix
\[ W_{n} = \begin{pmatrix}
A & B \\ B^T & D
\end{pmatrix} \]
where $A$ and $D$ are of size $\fl{n/2}\times\fl{n/2}$ and
$\ce{n/2}\times\ce{n/2}$, and the upper triangular entries
of $W_n$ are independent.
They let all entries of $W_n$ except the non-diagonal entries of $D$
be normal with mean $0$ and variance $1/n$, and simulated the spectrum
of $W_{n}$ for $n=2000$ to see that $\mu_{W_n}$ does not look
like a semicircle.
Note that \eqref{eq:row_one} does not hold.

Our main theorem will let us prove what was suggested by the simulation in
\cite{GNT15}, namely that $\E\mu_{W_n}\not\tto\musc$.
More generally, we will prove for a large class of Hermitian Wigner ensembles
$(W_n)_{n\in\N}$ that $\E\mu_{W_n}\tto\musc$ (or $\mu_{W_n}\tto\musc$ a.s.,
equivalently) holds \emph{if and only if} \eqref{eq:row_one} is true.

One thing we should notice is that changing $o(n)$ rows of $W_n$
has no effect on the limit of $\E\mu_{W_n}$ due to the following:

\begin{lemma}[rank inequality] \label{lem:rank_ineq}
Let $A$ and $B$ be $n \times n$ Hermitian matrices.
If $F_A$ and $F_B$ are the distribution functions of $\mu_A$ and $\mu_B$
(defined in the same way as $\mu_{W_n}$), then
\[ \sup_{x\in\R} \|F_A(x)-F_B(x)\| \le \frac{\rank(A-B)}{n}. \]
\end{lemma}

\begin{proof}
See \cite[Theorem A.43]{BS10}.
\end{proof}

We want to say that for certain Hermitian Wigner ensembles
$(W_n)_{n\in\N}$ with $\E\mu_{W_n}\tto\musc$, we have \eqref{eq:row_one}.
However, without further restriction on $(W_n)_{n\in\N}$, we can always change
$o(n)$ rows and columns of $W_n$ so that \eqref{eq:row_one} becomes false,
while leaving the limiting distribution of $\E\mu_{W_n}$ unchanged.
To avoid this problem, we assume that
\begin{equation} \label{eq:margin}
\frac{1}{n}\sum_{i\in J_n}\sum_{j=1}^n\E|\w|^2 \to 0
\qquad \text{for any $J_n\subset\{1,\ldots,n\}$ with $|J_n|/n\to0$.}
\end{equation}
Notice that this condition is weaker than \eqref{eq:row_bdd}.
If $W_n$ satisfies \eqref{eq:row_one} and \eqref{eq:margin}, and we change $o(n)$
rows and columns of it to obtain Hermitian $W_n'$ which also satisfies
\eqref{eq:margin}, then $W_n'$ also satisfies \eqref{eq:row_one}.
The following is our first main theorem:

\begin{theorem}[characterization of semicircle convergence] \label{thm:wigner_var}
Let $(W_n)_{n\in\N}$ be a Hermitian Wigner ensemble with $\E\w=0$
and $\E|\w|^2 < \infty$ for $i,j=1,\ldots,n$ satisfying
\eqref{eq:margin} and the Lindeberg condition \eqref{eq:lind}.
Then $\E\mu_{W_n} \tto \musc$ if and only if \eqref{eq:row_one} holds.
%\begin{equation} \tag{\ref{eq:row_one}}
%\frac{1}{n} \sum_{i=1}^n \Bigl| \sum_{j=1}^n \E|\w|^2 - 1 \Bigr| \to 0.
%\end{equation}
\end{theorem}

\begin{remark} \label{rem:wigner_var}
Since \eqref{eq:row_one} implies
\[
\frac{1}{n}\sum_{i\in J_n}\sum_{j=1}^n\E|\w|^2
\le \frac{1}{n}\sum_{i\in J_n}\biggl(\Bigl|\sum_{j=1}^n\E|\w|^2 - 1\Bigr|
+ 1\biggr) \to 0
\]
for any $J_n \subset \{1,\ldots,n\}$ with $|J_n|/n\to0$,
the sufficiency direction of Theorem~\ref{thm:wigner_var} does not
require \eqref{eq:margin}.
This proves the claim right after Theorem~\ref{thm:wigner_gnt}.

By Lindeberg's universality scheme \cite[Theorem 2]{GNT15} for random matrices,
it follows that we can remove the condition \eqref{eq:row_bdd}
(which is (5) in \cite{GNT15}) from the semicircle law \cite[Theorem 1]{GNT15}
for certain random symmetric matrices with dependent upper triangular entries.
\end{remark}

We can actually go beyond Theorem~\ref{thm:wigner_var} and allow
the entries of $W_n$ to have infinite variances,
for example when $w_{ij} = c_{ij}x_{ij}/\sqrt{n\log n}$ where
$x_{ij}$ has a density
\[
f(x) = \begin{cases}
1/|x|^3 & \text{if $|x| > 1$} \\
0 & \text{if $|x| \le 1$}
\end{cases}
\]
and $c_{ij}$ is a real number close to $1$.
To achieve this, instead of $\E\w=0$ and the Lindeberg condition \eqref{eq:lind}, we assume
\begin{equation} \label{eq:weak_zero}
\frac{1}{n} \sum_{i,j=1}^n \bigl(\E[\w;|\w|\le1]\bigr)^2 \to 0
\end{equation}
and
\begin{equation} \label{eq:weak_lind}
\frac{1}{n} \sum_{i,j=1}^n \P(|\w| > \eps) \to 0
\qquad \text{for all $\eps > 0$.}
\end{equation}
If $\E\w=0$ and \eqref{eq:lind} hold, then \eqref{eq:weak_zero} follows due to
\[ \bigl(\E[\w;|\w|\le1]\bigr)^2 = \bigl(\E[\w;|\w|>1]\bigr)^2
\le \E[|\w|^2;|\w|>1], \]
and \eqref{eq:weak_lind} follows by
\[ \P(|\w|>\eps) \le \eps^{-2}\E[|\w|^2;|\w|>\eps]. \]
Finally, \eqref{eq:margin} is replaced by
\begin{multline} \label{eq:weak_margin}
\frac{1}{n}\sum_{i\in J_n}\sum_{j=1}^n\E[|\w|^2;|\w|\le1] \to 0 \\
\text{for any $J_n\subset\{1,\ldots,n\}$ with $|J_n|/n\to0$.}
\end{multline}

The following is our second main theorem:

\begin{theorem}[characterization, general version] \label{thm:wigner}
Let $(W_n)_{n\in\N}$ be a Hermitian Wigner ensemble satisfying
\eqref{eq:weak_zero}, \eqref{eq:weak_lind}, and \eqref{eq:weak_margin}.
Then $\E\mu_{W_n} \tto \musc$ if and only if
\begin{equation} \label{eq:weak_row_one}
\frac{1}{n} \sum_{i=1}^n \Bigl|\sum_{j=1}^n
\E[|\w|^2;|\w|\le1]-1 \Bigr| \to 0.
\end{equation}
\end{theorem}

\begin{remark}
\begin{enumerate}
\item
As in Theorem~\ref{thm:wigner_var}, the sufficiency direction of
Theorem~\ref{thm:wigner} does not require \ref{eq:weak_margin}.
(See Remark \ref{rem:wigner_var}.)

\item
Theorem~\ref{thm:wigner_var} follows easily from Theorem~\ref{thm:wigner}.
See Lemma \ref{lem:var_reduct} in the appendix.

\item
In case the entries of $W_n$ are real, the sufficiency part
of Theorem~\ref{thm:wigner} can be easily proved
using Theorem~\ref{thm:wigner_gnt} and Lemma~\ref{lem:rank_ineq}.
This is covered in Section~\ref{sec:suff_sym}.

\item
Our full proof of the sufficiency is a careful consideration of
Wigner's moment method proof of the original semicircle law.
This is arguably more elementary than the proof of
Theorem~\ref{thm:wigner_gnt} in \cite{GNT15}, which first deals with
matrices with Gaussian entries using combinatorial arguments,
and then generalizes the result to symmetric Wigner ensembles using Lindeberg's universality scheme for random matrices.
\end{enumerate}
\end{remark}

The following corollary relates
$\E\mu_{W_n}\tto\musc$ to the convergence in distribution of the sum of a row
of $W_n$ to the standard normal random variable.
The sufficiency direction when the entries of $W_n$ are identically distributed
was covered by \cite{Jun18}. 
We denote the L\'evy metric by $L$.

\begin{corollary}[characterization, Gaussian convergence] \label{cor:wigner_gauss}
Under the hypotheses of Theorem~\ref{thm:wigner},
we have $\E\mu_{W_n}\tto\musc$ if and only if
\begin{equation} \label{eq:row_gauss}
\frac{1}{n} \sum_{i=1}^n L(F_{ni},G) \to 0,
\end{equation}
where $F_{ni}$ and $G$ are the distribution functions of $\sum_{j=1}^n \pm|\w|$
and the standard normal random variable.
The signs $\pm$ are independent Rademacher random variables independent
from $W_n$.
\end{corollary}

The rest of the paper is organized as follows.
Section \ref{sec:suff_sym} is a short section that introduces
Theorem~\ref{thm:red_wigner}.
This theorem is a reduction of Theorems~\ref{thm:wigner_var}
and~\ref{thm:wigner}, and will ultimately imply them as shown in
Appendix~\ref{app:reduct}.
This section also proves the sufficiency part of Theorem~\ref{thm:red_wigner}
in the case when the entries of $W_n$ are real.

Section \ref{sec:nec_six} is the essence of the proof of the necessity part
of Theorem~\ref{thm:red_wigner}.
We add an assumption that the sixth moment of $\E\mu_{W_n}$ is bounded,
but as a return we obtain a clean proof that is right to the point.
The idea is to express the second and the fourth moments of $\E\mu_{W_n}$
in terms of the variances of the entries of $W_n$.

Section~\ref{sec:lift_six} shows that we can remove the additional assumption
on the sixth moments, assuming that a certain lemma
(Lemma~\ref{lem:bdd_row_moment}) holds.
The condition \eqref{eq:margin} is used in this section.

Section~\ref{sec:comp_mom} proves Lemma~\ref{lem:bdd_row_moment}
by a systematic computation of the moments of $\E\mu_{W_n}$.
The computation is a variant of Wigner's original moment method, but it can
handle the case when the entries have non-identical variances.

In Section~\ref{sec:suff}, we prove the sufficiency part of
Theorem~\ref{thm:red_wigner} using the results of Section~\ref{sec:comp_mom}.
The classical argument involving Dyck paths is discussed for completeness.

Finally in Section \ref{sec:gauss}, we derive Corollary \ref{cor:wigner_gauss}
from Theorem \ref{thm:wigner} by an elementary argument involving
the Lindeberg--Feller central limit theorem.

Appendix \ref{app:exp_meas} defines the mean of a random probability measure,
and proves some of their properties that we need.
Appendix \ref{app:reduct} contains the fairly standard proof that
Theorem~\ref{thm:wigner} implies Theorem~\ref{thm:wigner_var}, and that
Theorem~\ref{thm:red_wigner} implies Theorem~\ref{thm:wigner}.

\section{Proof of sufficiency for symmetric Wigner ensembles}
\label{sec:suff_sym}

It is enough to prove the following in order to prove
Theorems \ref{thm:wigner_var} and \ref{thm:wigner}.
The justification for the reduction is fairly standard, and is covered
by Lemmas \ref{lem:var_reduct} and \ref{lem:reduct} in the appendix.

\begin{theorem}[characterization, reduced form] \label{thm:red_wigner}
Let $(W_n)_{n\in\N}$ be a Hermitian Wigner ensemble such that
\begin{multline} \label{eq:neg}
\E\w = 0 \quad\text{and}\quad |\w| \le \eps_n \qquad
\text{for all $n\in\N$ and $i,j=1,\ldots,n$} \\
\text{where}\quad
1 \ge \eps_1 \ge \eps_2 \ge \cdots \quad \text{and} \quad \eps_n \to 0,
\end{multline}
and \eqref{eq:margin} is true.
Then $\E\mu_{W_n} \tto \musc$ if and only if \eqref{eq:row_one}.
Note that \eqref{eq:lind}
is automatically satisfied due to $|\w| \le \eps_n$.
\end{theorem}

The following shows that we can assume \eqref{eq:row_bdd} in the proof of
the sufficiency part of Theorem \ref{thm:red_wigner}.

\begin{lemma} \label{lem:bdd_reduct}
Assume that $(W_n)_{n\in\N}$ satisfies \eqref{eq:row_one} and
the conditions of Theorem~\ref{thm:red_wigner}.
Then there exists a Hermitian Wigner ensemble
$(W'_n)_{n\in\N}$ satisfying \eqref{eq:row_one}, \eqref{eq:row_bdd}, and the
conditions of Theorem \ref{thm:red_wigner} such that $\mu_{W_n}\tto\musc$
a.s.\ if and only if $\mu_{W'_n}\tto\musc$ a.s.
\end{lemma}

\begin{proof}
If we let $J_n$ be the set of all $i \in \{1,\ldots,n\}$ such that
$\sum_{j=1}^n \E|\w|^2 > 2$, then
\[ \frac{|J_n|}{n} \le
\frac{1}{n} \sum_{i=1}^n \Bigl| \sum_{j=1}^n \E|\w|^2 - 1 \Bigr| \]
where the right side goes to $0$ by \eqref{eq:row_one}.
Let $W'_n = (\w')_{i,j=1}^n$ be given by $\w' = \w$ if $i,j\not\in J_n$
and $\w' = 0$ otherwise.

Then,
\[
\begin{split}
\frac{1}{n} \sum_{i=1}^n \Bigl| \sum_{j=1}^n \E|\w'|^2 - 1 \Bigr|
&\le \frac{1}{n} \sum_{i=1}^n \Bigl| \sum_{j=1}^n \E|\w|^2 - 1 \Bigr|
+ \frac{2}{n} \sum_{i\in J_n} \sum_{j=1}^n \E|\w|^2 \\
&\to 0
\end{split}
\]
where the right side goes to $0$ by \eqref{eq:row_one} and \eqref{eq:margin}.
Notice that $(W_n')_{n\in\N}$ satisfies \eqref{eq:row_bdd} for $C=2$.
The conditions of Theorem~\ref{thm:red_wigner} can be easily shown for
$(W_n')_{n\in\N}$.

Since
\[ \frac{\rank(W_n-W'_n)}{n} \le \frac{2|J_n|}{n} \to 0, \]
we have $\mu_{W_n}\tto\musc$ a.s.\ if and only if $\mu_{W'_n}\tto\musc$ a.s.\
by Lemma \ref{lem:rank_ineq}.
\end{proof}

If $W_n$ is real symmetric, then the sufficiency part of
Theorem~\ref{thm:red_wigner} follows from Lemma~\ref{lem:bdd_reduct} and
Theorem~\ref{thm:wigner_gnt}.
In Section \ref{sec:suff}, we will present a direct proof of
the sufficiency that applies to Hermitian Wigner ensembles and does not
depend on the result of \cite{GNT15}.

\section{Proof of necessity under bounded sixth moments} \label{sec:nec_six}

Assume \eqref{eq:neg} throughout this section.
In this section, we present a relatively simple proof of necessity
in Theorem~\ref{thm:red_wigner} under the following additional assumption:
\begin{equation} \label{eq:bdd_sixth}
\sup_{n\in\N} \int_\R x^6 \,\E\mu_{W_n} < \infty.
\end{equation}
The number $6$ comes out just because it is an even number greater than $4$.
Our proof is based on an examination of the second
and the fourth moments of $\E\mu_{W_n}$.
If $\lambda_1,\ldots,\lambda_n$ are the eigenvalues of $W_n$, then
for each $k \in \N$ we have
\[
\int_\R x^k \,\mu_{W_n}(dx) = \frac{1}{n}\sum_{i=1}^n \lambda_i^k
= \frac{1}{n}\tr W_n^k,
\]
and thus
\begin{equation} \label{eq:moment_trace}
\int_\R x^k \,\E\mu_{W_n}(dx) = \frac{1}{n}\E\tr W_n^k.
\end{equation}
(See Lemma \ref{lem:exp_change}.)

The second moment of $\E\mu_{W_n}$ can be easily expressed
in terms of the variances of $\w$.

\begin{lemma}[computation of the second moment] \label{lem:comp_two}
We have
\[ \int_\R x^2 \,\E\mu_{W_n}(dx) =
\frac{1}{n}\sum_{i,j=1}^n \E |w_{ij}|^2. \]
\end{lemma}

\begin{proof}
It follows from \eqref{eq:moment_trace} and
\[ \tr W_n^2 = \sum_{i,j=1}^n |w_{ij}|^2. \qedhere \]
\end{proof}

Computing the fourth moment requires more effort, but is still tractable.

\begin{lemma}[computation of the fourth moment] \label{lem:comp_four}
If
\begin{equation} \label{eq:bdd_var}
\sup_{n\in\N} \int_\R x^2 \,\E\mu_{W_n} < \infty,
\end{equation}
then
\[ \int_\R x^4 \,\E\mu_{W_n}(dx) - \frac{2}{n}\sum_{i=1}^n
\biggl( \sum_{j=1}^n \E|\w|^2 \biggr)^2 \to 0. \]
\end{lemma}

We remark that \eqref{eq:margin} implies \eqref{eq:bdd_var}.
The proof for that is similar to that of Lemma \ref{lem:almost_row_bdd} below.

\begin{proof}
Note that
\[ \int_\R x^4 \,\E\mu_{W_n}(dx) = \frac{1}{n}\E\tr W_n^4
= \frac{1}{n} \sum_{i_1,i_2,i_3,i_4=1}^n
\E[w_{i_1i_2}w_{i_2i_3}w_{i_3i_4}w_{i_4i_1}]. \]
Since the upper triangular entries of $W_n$ are independent and
have mean zero, in order for $\E[w_{i_1i_2}w_{i_2i_3}w_{i_3i_4}w_{i_4i_1}]$
not to vanish,
\[ \{i_1,i_2\},\{i_2,i_3\},\{i_3,i_4\},\{i_4,i_1\} \]
should either be all the same, or be partitioned into two groups,
where each group consists of two identical sets.
This implies either $i_1=i_3$ or $i_2=i_4$ or both.
Notice that, for instance, $i_1=i_2 \ne i_3=i_4$ cannot happen because $\{i_1,i_2\}$ would then appear only once among $\{i_1,i_2\}$, $\{i_2,i_3\}$, $\{i_3,i_4\}$, and $\{i_4,i_1\}$.

Thus, the sum on the right side equals
\[ \sum_{i,j,k=1}^n \E[|w_{ij}|^2|w_{ik}|^2] +
\sum_{i,j,k=1}^n \E[|w_{ij}|^2|w_{jk}|^2]
- \sum_{i,j=1}^n \E|w_{ij}|^4, \]
where the last term corresponds to the case where both $i_1=i_3$
and $i_2=i_4$ are true.
Since the first and the second sum both equal
\[ \sum_{i=1}^n\biggl(\sum_{j=1}^n\E|w_{ij}|^2\biggr)^2
+ \sum_{i,j=1}^n \bigl(\E|w_{ij}|^4 - (\E|w_{ij}|^2)^2 \bigr), \]
we have
\[
\int_\R x^4 \,\E\mu_{W_n}(dx) =
\frac{2}{n}\sum_{i=1}^n\biggl(\sum_{j=1}^n\E|w_{ij}|^2\biggr)^2
+ \frac{1}{n}\sum_{i,j=1}^n\bigl(\E|w_{ij}|^4 - 2(\E|w_{ij}|^2)^2\bigr).
\]

Since both $\E|\w|^4$ and $(\E|w_{ij}|^2)^2$ are bounded above by
$\eps_n^2\E|\w|^2$, we have
\[
\biggl|\frac{1}{n}\sum_{i,j=1}^n \bigl(\E|w_{ij}|^4 - 2(\E|w_{ij}|^2)^2\bigr)
\biggr| \le \frac{3\eps_n^2}{n} \sum_{i,j=1}^n \E|\w|^2 \to 0
\]
by \eqref{eq:bdd_var} and Lemma \ref{lem:comp_two}.
\end{proof}

Now we are ready to prove the necessity part of Theorem \ref{thm:red_wigner}
assuming \eqref{eq:bdd_sixth}.
Assume $\E\mu_{W_n} \tto \musc$.
By Skorokhod's representation theorem \cite[Theorem 25.6]{Bil12},
we can take real-valued random variables $X,X_1,X_2,\ldots$ on
a common probability space such that $\E\mu_{W_n}$ is the distribution
of $X_n$, $\musc$ is the distribution of $X$, and $X_n \to X$ a.s.

Since
\begin{equation} \label{eq:sup_six_bdd}
\sup_{n\in\N} \E X_n^6 = \sup_{n\in\N} \int_\R x^6 \,\E\mu_{W_n}(dx)
< \infty,
\end{equation}
$(X_n^2)_{n\in\N}$ and $(X_n^4)_{n\in\N}$ are uniformly integrable.
Thus, $X_n\to X$ a.s.\ implies $\E X_n^2 \to \E X^2$
and $\E X_n^4 \to \E X^4$.
By Lemma \ref{lem:comp_two} and Lemma \ref{lem:comp_four}, we have
\[
\frac{1}{n}\sum_{i,j=1}^{n}\E|\w|^2 \to
\E X^2 = \int_\R x^2\,\musc(dx) = 1
\]
and
\[
\frac{2}{n}\sum_{i=1}^n \biggl( \sum_{j=1}^n \E|\w|^2 \biggr)^2
\to \E X^4 = \int_\R x^4\,\musc(dx) = 2.
\]
See \cite[2.1.1]{AGZ10} for a computation of the moments of $\musc$.

Here is the punchline: using $(\E Y)^2 \le \E Y^2$
in the first line, and then applying the two convergences we have just
established, we have
\[
\begin{split}
\biggl(\frac{1}{n}\sum_{i=1}^n\Bigl|\sum_{j=1}^n \E|\w|^2-1\Bigr|\biggr)^2
&\le \frac{1}{n} \sum_{i=1}^n \biggl( \sum_{j=1}^n\E|\w|^2 - 1\biggr)^2\\
&= \frac{1}{n}\sum_{i=1}^n \biggl(\sum_{j=1}^n \E|\w|^2\biggr)^2
- \frac{2}{n}\sum_{i,j=1}^n \E|\w|^2 + 1 \\
&\to 1 - 2 + 1 = 0.
\end{split}
\]
Thus, the necessity in Theorem \ref{thm:red_wigner} is proved
under the assumption \eqref{eq:bdd_sixth}.

\section{Lifting the bounded sixth moment condition} \label{sec:lift_six}

In this section, we prove the necessity part of Theorem \ref{thm:red_wigner}
without assuming the bounded sixth moment condition \eqref{eq:bdd_sixth}.
We rely on the following lemma, which will be proved in the next section.

\begin{lemma}[bounded eighth moments] \label{lem:bdd_row_moment}
If \eqref{eq:row_bdd} and \eqref{eq:neg} hold,
then
\[ \sup_{n\in\N} \int_\R x^8 \,\E\mu_{W_n}(dx) < \infty. \]
\end{lemma}

The number $8$ is here just because it is even and greater than $6$.
In fact, our proof easily extends to any even natural number.
Given $K_n \subset \{1,\ldots,n\}$ for all $n\in\N$,
let $W_n^K$ be the matrix obtained from $W_n$ by replacing $\w$ with $0$
for all $(i,j) \not\in K_n\times K_n$.

\begin{lemma} \label{lem:almost_row_bdd}
Assume \eqref{eq:margin}.
For any $\eps > 0$, there exist $K_n \subset \{1,\ldots,n\}$ with
$|K_n| \ge (1-\eps)n$ such that
\[ \sum_{j=1}^n \E|\w|^2 \le C \qquad \text{for all $n\in\N$
and $i \in K_n$} \]
for some $C < \infty$.
\end{lemma}

\begin{proof}
We may assume $\eps < 1$.
Suppose that the claim is false, and let $K_n$ be the set of size
$\ce{(1-\eps)n}$ consisting of
$i \in \{1,\ldots,n\}$ with $\ce{(1-\eps)n}$ smallest $\sum_{j=1}^n \E|\w|^2$.
Notice that
\[ \min_{i\in K_n^c} \sum_{j=1}^n \E|\w|^2 \ge
\max_{i\in K_n} \sum_{j=1}^n \E|\w|^2 \to \infty \]
along some subsequence $(n_k)_{k\in\N}$, where $K_n^c:=\{1,\ldots,n\}\setminus
K_n$.
For all $n=n_1,n_2,\ldots$ such that the left side in the previous display is
at least $1$, let $J_n$ be any subset of $K_n^c$ such that
\[ |J_n| = \biggl\lceil\frac{\eps n}{\min_{i\in K_n^c}
\sum_{j=1}^n \E|\w|^2}\biggr\rceil. \]
Then $|J_{n_k}|/n_k \to 0$ follows from $\ce{x} \le x+1$.
However,
\[ \frac{1}{n}\sum_{i\in J_n} \sum_{j=1}^n \E|\w|^2
\ge \frac{|J_n|}{n} \min_{i\in K_n^c} \sum_{j=1}^n \E|\w|^2 \ge \eps \]
for $n=n_1,n_2,\ldots$ for which $J_n$ is defined.
If we let $J_n := \emptyset$ for all $n$ for which $J_n$ is undefined,
then $J_n$ contradicts \eqref{eq:margin}.
\end{proof}

Given $K_n \subset \{1,\ldots,n\}$ for each $n\in\N$,
let $W_n^K$ be the matrix obtained from $W_n$ by replacing $\w$ with $0$
for all $(i,j) \not\in K_n\times K_n$.

\begin{lemma} \label{lem:almost_six_bdd}
If \eqref{eq:margin} and \eqref{eq:neg} hold, then there are
$K_n \subset \{1,\ldots,n\}$ with $|K_n|/n \to 1$ such that
\begin{equation} \label{eq:trim_bdd_sixth}
\sup_{n\in\N} \int_\R x^6 \,\E\mu_{W_n^K}(dx) < \infty.
\end{equation}
\end{lemma}

\begin{proof}
Let $\eps \in (0,1)$, and $K_n$ and $C$ be as in the preceding lemma.
Suppose that
\[ \int_\R x^6 \,\E\mu_{W_{n_k}^K}(dx) \to c > 3^6 \qquad
\text{as $k\to\infty$} \]
for some $n_1<n_2<\cdots$.
By Lemmas \ref{lem:bdd_row_moment} and \ref{lem:almost_row_bdd},
\[ \sup_{n\in\N} \int_\R x^8\,\E\mu_{W_n^K}(dx) < \infty. \]
This implies that $(\E\mu_{W_{n_k}^K})_{k\in\N}$ is tight,
thus it has a subsequence weakly convergent to some $\mu$,
which we still, by abuse of notation, denote by
$(\E\mu_{W_{n_k}^K})_{k\in\N}$.
By Skorokhod's theorem and the uniform integrability argument
that followed \eqref{eq:sup_six_bdd}, we have
\[ \int_\R x^6 \,\mu(dx) = \lim_{k\to\infty} \int_\R x^6
\,\E\mu_{W_{n_k}^K}(dx) > 3^6, \]
and thus $\mu([-3,3]^c) > 0$.

If $W_n^K$ has $k$ eigenvalues outside $[-3,3]$,
then the Cauchy interlacing law \cite[Exercise 1.3.14]{Tao12}
implies that $W_n$ has at least $k$ eigenvalues outside $[-3,3]$.
Thus
\[ \mu_{W_n}([-3,3]^c) \ge (1-\eps)\mu_{W_n^K} ([-3,3]^c), \]
and therefore
\[ \E\mu_{W_n}([-3,3]^c) \ge (1-\eps)\E\mu_{W_n^K}([-3,3]^c) \]
by Lemma \ref{lem:exp_change}.
Since $\E\mu_{W_{n_k}} \tto \musc$, the portmanteau theorem implies
\[
\begin{split}
\musc((-3,3)^c) &\ge \limsup_{k\to\infty} \E\mu_{W_{n_k}}((-3,3)^c) \\
&\ge (1-\eps) \limsup_{k\to\infty} \E\mu_{W_{n_k}^K}((-3,3)^c) \\
&\ge (1-\eps) \liminf_{k\to\infty} \E\mu_{W_{n_k}^K}([-3,3]^c) \\
&\ge (1-\eps) \mu([-3,3]^c) > 0,
\end{split}
\]
but this contradicts the fact that $\musc$ is supported on $[-2,2]$.
Thus, we have
\[ \limsup_{n\to\infty} \int_\R x^6 \,\E\mu_{W_n^K}(dx) \le 3^6. \]

Since $\eps > 0$ is arbitrary, we have actually proved that for each $\eps>0$
we can choose $K_n^\eps \subset \{1,\ldots,n\}$ such that
$|K^\eps_n| \ge (1-\eps)n$ and
\[ \limsup_{n\to\infty} \int_\R x^6 \,\E\mu_{W_n^{K^\eps}}(dx) \le 3^6. \]
Choose positive integers $m_1 < m_2 < \cdots$ so that
\[ \int_\R x^6 \,\E\mu_{W_n^{K^{1/k}}} \le 3^6+1
\qquad \text{for all $n\ge m_k$,} \]
and let $K_n := K_n^{1/k}$ for $n=m_k,\ldots,m_{k+1}-1$
and $K_n := \emptyset$ for $n=1,\ldots,m_1-1$.
(We are redefining $K_n$ by abuse of notation.)
Then we have $|K_n|/n \to 1$ and \eqref{eq:trim_bdd_sixth}.
\end{proof}

We are ready to prove the necessity part of Theorem~\ref{thm:red_wigner}.
Assume \eqref{eq:margin}, \eqref{eq:neg}, and $\E\mu_{W_n}\tto\musc$.
By Lemma \ref{lem:as_reduct}, we have $\mu_{W_n} \tto \musc$ a.s.
If $K_1,K_2,\ldots$ are as in Lemma~\ref{lem:almost_six_bdd}, then
\[ \frac{\rank(W_n - W_n^K)}{n} \le \frac{2(n-|K_n|)}{n} \to 0, \]
and thus $\mu_{W_n^K} \tto \musc$ a.s.\ by Lemma~\ref{lem:rank_ineq}.
By another application of Lemma \ref{lem:as_reduct}, $\E\mu_{W_n^K}\tto\musc$.
As we have \eqref{eq:trim_bdd_sixth}, the previous section tells us that
\[
\frac{1}{n}\sum_{i\in K_n}\Bigl|\sum_{j\in K_n}\E|\w|^2-1\Bigr| \to 0.
\]
Since $|K_n^c|/n \to 0$, the assumption \eqref{eq:margin} implies
\eqref{eq:row_one}.
Thus, the necessity part of Theorem~\ref{thm:red_wigner} is proved assuming that
Lemma \ref{lem:bdd_row_moment} holds.

\section{Computation of moments} \label{sec:comp_mom}

The goal of this section is to prove Lemma~\ref{lem:bdd_row_moment}
and also establish some arguments needed in the next section.
We use a variant of Wigner's original moment method that can handle
entries with non-identical variances.
Those that are very familiar with these arguments may want to
jump ahead to the proof of Lemma~\ref{lem:bdd_row_moment}.

Assume \eqref{eq:row_bdd} and \eqref{eq:neg} throughout this section.
Recall that
\begin{equation} \tag{\ref{eq:moment_trace}}
\int_\R x^k \,\E\mu_{W_n}(dx) = \frac{1}{n}\E\tr W_n^k
\end{equation}
for all $k \in \N$.
In this section, we compute the asymptotics of $n^{-1}\E\tr W_n^k$
as $n\to\infty$.

Fix $k \in \N$.
The boldface lower case letters $\bi, \bj, \ldots$ will denote
$(i_0,\ldots,i_k)$, $(j_0,\ldots,j_k)$, and so on.
Let us call a $(k+1)$-tuple $\bi$ with $i_0=i_k$
a \emph{closed walk} of length $k$.
For any closed walk $\bi$ with $i_0,\ldots,i_k \in \{1,\ldots,n\}$, let
\[ w_\bi := \prod_{s=1}^k w_{i_{s-1}i_s}. \]
Notice that
\[ \frac{1}{n}\E\tr W_n^k = \frac{1}{n}\sum_{\bi} \E w_\bi  \]
where $\bi$ ranges over all closed walks $\bi$ (of length $k$) with
$i_0,\ldots,i_k \in \{1,\ldots,n\}$.

Now we gather together the closed walks which have the same ``shape."
Let us say that two closed walks $\bi$ and $\bj$ are \emph{isomorphic}
if for any $s,t = 0,\ldots,k$ we have $i_s=i_t$ if and only if $j_s=j_t$.
A \emph{canonical closed walk} of length $k$ on $t \in \N$ vertices
is a closed walk $\bc$ such that
\begin{enumerate}
\item
$c_0=c_k=1$,
\item
$\{c_0,\ldots,c_k\}=\{1,\ldots,t\}$, and
\item
$c_i \le \max\{c_0,\ldots,c_{i-1}\}+1$ for each $i=1,\ldots,k$.
\end{enumerate}
Let $\gamma(k,t)$ denote the set of such walks.
It is straightforward to show that any closed walk is isomorphic to
exactly one canonical closed walk.
For each $\bc \in \gamma(k,t)$, let $L(n,\bc)$ denote the set of all closed
walks $\bi$ with $i_0,\ldots,i_k \in \{1,\ldots,n\}$
which are isomorphic to $\bc$.
Then we have
\begin{equation} \label{eq:moment_exp}
\frac{1}{n}\E\tr W_n^k = \frac{1}{n} \sum_{i_0,\ldots,i_k=1}^n
\E w_\bi
= \frac{1}{n} \sum_{t=1}^{k} \sum_{\bc \in \gamma(k,t)}
\sum_{\bi \in L(n,\bc)} \E w_\bi,
\end{equation}
where the upper bound of $t$ is (rather arbitrarily) set to $k$
since $\gamma(k,t)$ is empty for all $t > k$.

We will fix $t \in \N$ and $\bc \in \gamma(k,t)$, and
compute $n^{-1}\sum_{\bi \in L(n,\bc)} \E w_\bi$.
As a first step, we get an easy case out of the way.

\begin{lemma}[zeroed out terms] \label{lem:once_zero}
If $\bc$ crosses some edge $\{i,j\}$ exactly once, i.e.,
$\{c_{s-1},c_s\} = \{i,j\}$ for exactly one $s \in \{1,\ldots,k\}$, then
$\E w_\bi = 0$ for any $n \in \N$ and $\bi \in L(n,\bc)$.
\end{lemma}

\begin{proof}
Since we have \eqref{eq:neg} and the upper triangular entries of $W_n$
are independent,
$w_\bi$ is the product of $w_{ij}$ (or $w_{ji}$) and a bounded random variable
which is independent from $w_{ij}$.
Since $\E w_{ij} = 0$, we have $\E w_\bi = 0$.
\end{proof}

Now assume that $\bc$ does not cross any edge exactly once, i.e.,
for each $s \in \{1,\ldots,k\}$ there is some $r \in \{1,\ldots,k\}$
distinct from $s$ such that $\{c_{s-1},c_s\} = \{c_{r-1},c_r\}$.
To compute $n^{-1}\sum_{i\in L(n,\bc)}\E w_\bi$, we introduce some notation.
Let $G(\bc)$ be the graph (possibly having loops but no multiple edges)
with the vertex set
\[ \{i_0,\ldots,i_k\} \]
and the edge set
\[ \bigl\{\{i_{t-1},i_t\} \bigm| t=1,\ldots,k\bigr\}. \]

For a tree $T$, let $V(T)$ and $E(T)$ denote the vertex set and
the edge set of $T$.
Given a finite tree $T$ and $n \in \N$, we let $I(T,n)$ denote
the set of injections from the vertex set $V(T)$ of $T$ to $\{1,\ldots,n\}$.
For each $F \in I(T,n)$, let us write
\[
\Pi(F) := \prod_{e\in E(T)} \E|w_{F(x_e)F(y_e)}|^2
\]
where $x_e,y_e \in V(T)$ denote the endpoints of $e \in E(T)$.
We are omitting the dependence of $\Pi(F)$ on $T$, but there should be
no confusion.
Note that the value $\Pi(F)$ is well-defined because $W_n$ is Hermitian.

Now we get back to the problem of computing
$n^{-1}\sum_{i\in L(n,\bc)}\E w_\bi$.
As each edge of $G(\bc)$ is crossed at least twice by $\bc$,
there are at most $k/2$ edges in $G(\bc)$.
As $G(\bc)$ is a connected graph with $t$ vertices, we have $t\le k/2+1$,
and $G(\bc)$ has a spanning tree $S$ with $t-1$ edges. Choose some $S$.
For each $\bi \in L(n,\bc)$, consider the injection
$F_\bi\colon \{1,\ldots,t\} \to \{1,\ldots,n\}$ given by
$F_\bi(c_s) := i_s$ for each $s=0,\ldots,k$.

First assume $t = k/2+1$.
Since $S$ has $k/2$ edges and each edge of $S$ is crossed twice by $\bc$,
we have $S = G(\bc)$.
As each edge of $G(\bc)$ is traversed exactly once in each direction, and
the map $L(n,\bc) \to I(S,n)$ given by $\bi \mapsto F_\bi$ is a bijection,
we have
\begin{equation} \label{eq:tree_sum}
\frac{1}{n} \sum_{\bi\in L(n,\bc)} \E w_\bi
= \frac{1}{n} \sum_{\bi\in L(n,\bc)} \Pi(F_\bi)
= \frac{1}{n} \sum_{F \in I(S,n)} \Pi(F).
\end{equation}

Now assume $t < k/2+1$. By $|w_{ij}| \le \eps_n$ and the fact that
$\bc$ crosses any edge of $G(\bc)$ at least twice, we have
\[ |\E w_\bi| \le \eps_n^{k-2(t-1)} \Pi(F_\bi). \]
Note that $\eps_n^{k-2(t-1)} \to 0$ since $t < k/2 + 1$.
By using the bijection $L(n,\bc) \to I(S,n)$ again, we have

\begin{equation} \label{eq:nontree_sum}
\begin{split}
\frac{1}{n}\sum_{\bi\in L(n,\bc)} |\E w_\bi|
&\le \frac{\eps_n^{k-2(t-1)}}{n} \sum_{\bi \in L(n,\bc)} \Pi(F_\bi) \\
&= \frac{\eps_n^{k-2(t-1)}}{n} \sum_{F \in I(S,n)} \Pi(F).
\end{split}
\end{equation}
The right side tends to $0$ by the following.

\begin{lemma}[contribution of a tree] \label{lem:tree_bdd}
If $T$ is a finite tree with $m$ edges, $x \in V(T)$, $n \in \N$,
and $i \in \{1,\ldots,n\}$, then
\begin{equation} \label{eq:tree_bdd}
\sum_{\substack{F \in I(T,n) \\ F(x)=i}} \Pi(F) \le C^m
\end{equation}
where $C$ is as in \eqref{eq:row_bdd}.
Note that it follows that
\[ \frac{1}{n} \sum_{F\in I(T,n)} \Pi(F) \le C^m. \]
\end{lemma}

\begin{proof}
There is nothing to prove if $m=0$.
To proceed by induction, assume that \eqref{eq:tree_bdd} holds,
and let $T$ be a tree with $m+1$ edges.
Pick any leaf $y \in V(T)$ distinct from $x$, and let $z \in V(T)$ be the only
vertex that is adjacent to $y$.
Since
\[
\begin{split}
\sum_{\substack{F\in I(T,n) \\ F(x)=i,F(z)=j}} \Pi(F)
&\le \sum_{\substack{H\in I(T\setminus y,n) \\ H(x)=i,H(z)=j}}
\biggl( \Pi(H)\sum_{l=1}^n \E |w_{jl}|^2 \biggr) \\
&\le C\sum_{\substack{H\in I(T\setminus y,n) \\ H(x)=i,H(z)=j}} \Pi(H)
\end{split}
\]
by \eqref{eq:row_bdd} for all $j=1,\ldots,n$, we have
\[
\begin{split}
\sum_{\substack{F\in I(T,n) \\ F(x)=i}} \Pi(F)
&= \sum_{j=1}^n \sum_{\substack{F\in I(T,n) \\ F(x)=i,F(z)=j}} \Pi(F) \\
&\le C\sum_{j=1}^n \sum_{\substack{H\in I(T\setminus y,n) \\ H(x)=i,H(z)=j}}
\Pi(H)
= C\sum_{\substack{H\in I(T\setminus y,n) \\ H(x)=i}} \Pi(H) \le C^{m+1}
\end{split}
\]
by the induction hypothesis.
\end{proof}

We have now shown the following; see \eqref{eq:tree_sum} and
\eqref{eq:nontree_sum}.

\begin{lemma}[contribution of a canonical walk] \label{lem:class_conv}
Let $t \in \N$ and $\bc \in \gamma(k,t)$.
Assume that $\bc$ does not cross any edge exactly one, i.e.,
for each $s \in \{1,\ldots,k\}$ there is some $r \in \{1,\ldots,k\}$
distinct from $s$ such that $\{c_{s-1},c_s\} = \{c_{r-1},c_r\}$.
Then we have $t \le k/2+1$.
If $t < k/2 + 1$, we have
\[
\frac{1}{n} \sum_{\bi\in L(n,\bc)} \E w_\bi \to 0.
\]
If $t = k/2 + 1$, then $G(\bc)$ is a tree, and we have
\[
\frac{1}{n} \sum_{\bi\in L(n,\bc)} \E w_\bi
= \frac{1}{n} \sum_{F \in I(G(\bc),n)} \Pi(F).
\]
\end{lemma}

Combining \eqref{eq:moment_trace}, \eqref{eq:moment_exp},
Lemma \ref{lem:once_zero}, and Lemma \ref{lem:class_conv},
we obtain the following approximation to the moments of $\E\mu_{W_n}$.

\begin{lemma}[computation of moments] \label{lem:moment_approx}
Let $\Gamma_k$ be the set of all $\bc \in \gamma(k,k/2+1)$
which crosses each edge of $G(\bc)$ twice.
(Note that $G(\bc)$ should be a tree, and that $\Gamma_k$ is finite.) Then,
\[ \int_\R x^k \,\E\mu_{W_n}(dx) - \frac{1}{n}\sum_{\bc \in \Gamma_k}
\sum_{F \in I(G(\bc),n)} \Pi(F) \to 0. \]
\end{lemma}

We now prove Lemma \ref{lem:bdd_row_moment} as promised.

\begin{proof}[Proof of Lemma \ref{lem:bdd_row_moment}]
Let $\Gamma_k$ be as in Lemma~\ref{lem:moment_approx}.
By Lemma \ref{lem:tree_bdd}, we have
\[ \frac{1}{n} \sum_{F\in I(G(\bc),n)} \Pi(F) \le C^4
\qquad \text{for all $\bc \in \Gamma_8$.} \]
Since $\Gamma_8$ is finite, we have
\[ \sup_{n\in\N} \biggl( \frac{1}{n} \sum_{\bc\in \Gamma_8}
\sum_{F\in I(G(\bc),n)} \Pi(F) \biggr) < \infty, \]
and thus
\[ \sup_{n\in\N} \int_\R x^8 \,\E\mu_{W_n}(dx) < \infty \]
by Lemma \ref{lem:moment_approx}.
\end{proof}

\section{Proof of sufficiency} \label{sec:suff}

We continue to use the notation introduced in Section~\ref{sec:comp_mom}.
On top of \eqref{eq:row_one} and \eqref{eq:neg}, we assume \eqref{eq:row_bdd},
which is possible by Lemma \ref{lem:bdd_reduct}.
We are one lemma away from proving the sufficiency part of
Theorem \ref{thm:red_wigner}.
Our proof is essentially a manifestation of Wigner's original idea.

\begin{lemma}[each tree contributes one] \label{lem:tree_one}
For any finite tree $T$ we have
\begin{equation} \label{eq:tree_one}
\lim_{n\to\infty} \frac{1}{n} \sum_{F\in I(T,n)} \Pi(F) = 1.
\end{equation}
(Compare it to Lemma \ref{lem:tree_bdd}.)
\end{lemma}

\begin{proof}
There is nothing to prove if $T$ has no edges.
To proceed by induction, assume that \eqref{eq:tree_one} holds
if $T$ has $m$ edges, and let $T$ be a tree with $m+1$ edges.
Let $x \in V(T)$ be a leaf of $T$, and $y$ be the only vertex of $T$
that is adjacent to $x$. Note that
\[
\begin{split}
\biggl| \frac{1}{n} \sum_{F\in I(T,n)} \Pi(F)
- \frac{1}{n}\sum_{H\in I(T\setminus x,n)} &\biggl( \Pi(H)\sum_{i=1}^n
\E |w_{H(y)i}|^2 \biggr) \biggr| \\
&\le \frac{(m+1)\eps_n^2}{n} \sum_{H\in I(T\setminus x,n)} \Pi(H)
\to 0
\end{split}
\]
by $|w_{ij}| \le \eps_n$ and Lemma \ref{lem:tree_bdd}
(or the induction hypothesis).
Applying Lemma~\ref{lem:tree_bdd} once again, we have
\[
\begin{split}
\biggl| \frac{1}{n} \sum_{H\in I(T\setminus x,n)}
&\biggl( \Pi(H) \sum_{i=1}^n \biggl( \E|w_{H(y)i}|^2 - \frac{1}{n} \biggr)
\biggr) \biggr| \\
&= \frac{1}{n} \biggl| \sum_{j=1}^n \biggl[ \sum_{i=1}^n
\biggl(\E |w_{ji}|^2 - \frac{1}{n} \biggr) \cdot
\sum_{\substack{H\in I(T\setminus x,n)\\ H(y)=j}} \Pi(H) \biggr] \biggr| \\
&\le \frac{C^m}{n} \sum_{j=1}^n \biggl| \sum_{i=1}^n
\biggl(\E |w_{ji}|^2 - \frac{1}{n}\biggr)\biggr| \to 0
\end{split}
\]
by \eqref{eq:row_one}.
The claimed result follows from the previous two displays and
\[ \lim_{n\to\infty} \frac{1}{n} \sum_{H\in I(T\setminus x,n)} \Pi(H) = 1.
\qedhere \]
\end{proof}

By Lemma \ref{lem:moment_approx} and Lemma \ref{lem:tree_one}, we have
\[ \int_\R x^k\,\E\mu_{W_n}(dx) \to |\Gamma_k|. \]
If $k$ is odd, then $|\Gamma_k| = 0$ because $k/2+1$ is not an integer.

Assume that $k$ is even.
A \emph{Dyck path} of length $k$ is a finite sequence $(x_0,\ldots,x_k)$
satisfying
\begin{enumerate}
\item
$x_0=x_k=0$,
\item
$|x_s-x_{s-1}|=1$ for all $s=1,\ldots,k$, and
\item
$x_s \ge 0$ for all $s=0,\ldots,k$.
\end{enumerate}
Given $\bc \in \Gamma_k$, let $D(\bc) := (x_0,\ldots,x_k)$ where
$x_s$ is the distance between $c_0$ and $c_s$ in $G(\bc)$.
Then it is clear that $D(\bc)$ is indeed a Dyck path, and it is not difficult
to see that $D$ is a bijection from $\Gamma_k$ to the set of all Dyck paths of
length $k$.
It is well-known that there are exactly $\frac{1}{k/2+1}\binom{k}{k/2}$
Dyck paths of length $k$; see \cite[Example 14.8]{vLW01}.
Thus, we have $|\Gamma_k| = \frac{1}{k/2+1}\binom{k}{k/2}$.

A direct computation (see \cite[2.1.1]{AGZ10}) yields
\[ \int_\R x^k\,\musc(dx)=\frac{1}{k/2+1}\binom{k}{k/2}
\qquad \text{for all even $k \in \N$,} \]
where the odd moments of $\musc$ are all zero due to the symmetry.
Thus,
\[ \int_\R x^k\,\E\mu_{W_n}(dx) \to \int_\R x^k\,\musc(dx)
\qquad \text{for all $k \in \N$.} \]
Since
\[ \biggl|\sum_{k=1}^\infty \frac{1}{k!}\int_\R x^k\,\musc(dx)\,r^k\biggr|
\le \sum_{k=1}^\infty \frac{|2r|^k}{k!} < \infty \]
by the ratio test for all $r\in\R$, the probability measure $\musc$
is determined by its moments by \cite[Theorem 30.1]{Bil12}.
Therefore, the moment convergence theorem \cite[Theorem 30.2]{Bil12}
tells us that $\E\mu_{W_n} \tto \musc$.

\section{Gaussian convergence} \label{sec:gauss}

Assume that $(W_n)_{n\in\N}$ satisfies the conditions of
Theorem~\ref{thm:wigner}.
In this section, we prove Corollary \ref{cor:wigner_gauss} by
showing that \eqref{eq:weak_row_one} and \eqref{eq:row_gauss} are equivalent.
We need the following two simple facts.

\begin{lemma}[converging averages] \label{lem:average}
For each $n \in \N$, let $x_{1}^{(n)},\ldots,x_{n}^{(n)} \ge 0$.
If
\[ \frac{1}{n} \sum_{i=1}^n x_{i}^{(n)} \to 0, \]
then we can take nonempty $K_n \subset \{1,\ldots,n\}$ for each $n \in \N$
so that
\[ \frac{|K_n|}{n} \to 1 \qquad \text{and} \qquad
\max_{i\in K_n} |x_i^{(n)}| \to 0. \]
\end{lemma}

\begin{proof}
For each $\eps > 0$, we have
\[
\frac{|\{i : x_i^{(n)} > \eps\}|}{n} \le
\frac{1}{n\eps}\sum_{i=1}^n x_i^{(n)} \to 0.
\]
We can take positive $\eps_1,\eps_2,\ldots$
with $\eps_n \to 0$ such that
\[
\frac{|\{i : x_i^{(n)} > \eps_n\}|}{n} \to 0.
\]
Let $K_n$ be $\{i : x_i^{(n)} \le \eps_n\}$ if it is nonempty, and let $K_n := \{1\}$ otherwise.
\end{proof}

\begin{lemma}[uniform convergence] \label{lem:uniform}
Let $(E,d)$ be a metric space, $A_1, A_2, \ldots \subset E$, and $x \in E$.
Then the following are equivalent:
\begin{enumerate}
\item
$x_n \to x$ for any choice of $x_1\in A_1, x_2\in A_2,\ldots\,$.
\item
$\sup_{y\in A_n} d(x,y) \to 0$.
\end{enumerate}
\end{lemma}

\begin{proof}
We omit the easy proof.
\end{proof}

By \eqref{eq:weak_lind} and Lemma \ref{lem:average}, for each $\eps>0$ we have
nonempty $K_n^\eps \subset \{1,\ldots,n\}$ with
\[ \frac{|K_n^\eps|}{n} \to 1 \qquad \text{and} \qquad
\max_{i\in K_n^\eps} \sum_{j=1}^n \P(|\w| > \eps) \to 0. \]
We can take $\eps_1 \ge \eps_2 \ge \cdots$ with
$\eps_n \to 0$ such that
\begin{equation} \label{eq:almost_null}
\frac{|K_n^{\eps_n}|}{n} \to 1 \qquad \text{and} \qquad
 \max_{i\in K_n^{\eps_n}} \sum_{j=1}^n \P(|\w| > \eps_n) \to 0.
\end{equation}

First assume \eqref{eq:weak_row_one}.
By Lemma \ref{lem:average}, we have nonempty $K_n \subset \{1,\ldots,n\}$
such that
\[ \frac{|K_n|}{n} \to 1 \qquad \text{and} \qquad
\max_{i\in K_n} \Bigl| \sum_{j=1}^n \E[|\w|^2;|\w|\le1] - 1 \Bigr| \to 0. \]
By \eqref{eq:almost_null}, we can make $K_n$ smaller so that
\[ \max_{i\in K_n} \sum_{j=1}^n \P(|\w|>\eps) \to 0 \]
also holds while retaining $|K_n|/n \to 1$.

Let $i_n \in K_n$ for each $n \in \N$.
By Lemma \ref{lem:uniform}, we have
\[ \sum_{j=1}^n \E[|w_{i_nj}|^2;|w_{i_nj}|\le1] \to 1
\qquad\text{and}\qquad \sum_{j=1}^n \P(|w_{i_nj}|>\eps_n) \to 0. \]
Since
\[ \sum_{j=1}^n \E[|w_{i_nj}|^2; \eps < |w_{i_nj}| \le 1]
\le \sum_{j=1}^n \P(|w_{i_nj}| > \eps) \to 0 \]
for all $\eps > 0$, the Lindeberg--Feller central limit theorem
\cite[Theorem 5.12]{Kal02} implies
\[ \sum_{j=1}^n \pm |w_{i_nj}|\I\{|w_{i_nj}|\le1\} \tto Z \]
where $Z$ is standard normal.
As $\sum_{j=1}^n \P(|w_{i_nj}| > 1) \to 0$, it follows that
$\sum_{j=1}^n \pm |w_{i_nj}| \tto Z$.

Since $i_1,i_2,\ldots$ are arbitrary, Lemma \ref{lem:uniform} implies
\[ \max_{i\in K_n} L(F_{ni_n},G) \to 0. \]
As the L\'evy distance is bounded above by $1$, we conclude that
\[ \frac{1}{n}\sum_{i=1}^n L(F_{ni},G)
\le \max_{i\in K_n} L(F_{ni_n},G) + \frac{n-|K_n|}{n} \to 0. \]

Now we assume \eqref{eq:row_gauss}.
By Lemma \ref{lem:average} and \eqref{eq:almost_null}, we have
nonempty $K_n \subset \{1,\ldots,n\}$ with $|K_n|/n \to 1$,
\[ \max_{i\in K_n} L(F_{ni},G) \to 0, \qquad \text{and} \qquad
\max_{i\in K_n} \sum_{j=1}^n \P(|\w| > \eps_n) \to 0. \]
Let $i_n \in K_n$ for each $n \in \N$.
By Lemma \ref{lem:uniform}, we have
\[ \sum_{j=1}^n \pm|w_{i_nj}| \tto Z \qquad\text{and}\qquad
\sum_{j=1}^n\P(|w_{i_nj}|>\eps_n) \to 0. \]
Since $\sum_{j=1}^n \P(|w_{i_nj}| > 1) \to 0$, we have
\begin{equation} \label{eq:trun_gauss}
\sum_{j=1}^n \pm|w_{i_nj}|\I\{|w_{i_nj}|\le1\} \tto Z.
\end{equation}

Let
\[ c_n := \sum_{j=1}^n \E[|w_{i_nj}|^2;|w_{i_nj}|\le1]. \]
If $c_n \to 0$ along some subsequence, then
\[ \E\Bigl(\sum_{j=1}^n \pm|w_{i_nj}|\I\{|w_{i_nj}|\le1\}\Bigr)^2 \to 0 \]
along that subsequence, but it contradicts \eqref{eq:trun_gauss}.
If $c_n \to c \in (0,\infty]$ along some subsequence, then
\[ \frac{1}{\sqrt{c_n}}\sum_{j=1}^n \pm|w_{i_nj}|\I\{|w_{i_nj}|\le1\} \tto Z \]
along that subsequence by the Lindeberg--Feller central limit theorem,
and so $c = 1$.
Thus, we have $c_n \to 1$.

Since $i_1,i_2,\ldots$ are arbitrary, Lemma \ref{lem:uniform}
and \eqref{eq:weak_margin} imply
\[
\begin{split}
\frac{1}{n} \sum_{i=1}^n \Bigl|\sum_{j=1}^n \E[|\w|^2&;|\w|\le1]-1 \Bigr| \\
&\le \max_{i\in K_n} \Bigl|\sum_{j=1}^n \E[|\w|^2;|\w|\le1]-1 \Bigr| \\
&\quad + \frac{1}{n} \sum_{i\in K_n^c} \sum_{j=1}^n \E[|\w|^2;|\w|\le1]
+ \frac{n-|K_n|}{n} \to 0.
\end{split}
\]

\appendix
\section{Mean probability measures} \label{app:exp_meas}

In this section, we clarify what we mean by $\E\mu_{W_n}$, and prove that
$\E\mu_{W_n} \tto \musc$ is equivalent to $\mu_{W_n} \tto \musc$ a.s.\
if $(W_n)_{n\in\N}$ is a Hermitian Wigner ensemble.

Let $\Pr(\R)$ be the set of all Borel probability measures on $\R$.
Equip $\Pr(\R)$ with the smallest $\sigma$-field that makes
$\Pr(\R) \ni \mu \mapsto \mu(-\infty,x]$ measurable for all $x \in \R$.

For any random element $\mu$ of $\Pr(\R)$, it is straightforward to show that
$\R \ni x \mapsto \E\mu(-\infty,x]$ is a distribution function
of some Borel probability measure on $\R$.
Let $\E\mu$ denote that measure. Then $\E\mu$ has the following property:

\begin{lemma}[change of order] \label{lem:exp_change}
Let $\mu$ be a random element of $\Pr(\R)$, and $f \colon \R \to \R$
be (Borel) measurable.
\begin{enumerate}
\item
If $f$ is nonnegative, then $\int_\R f\,d\mu$ is measurable, and we have
\begin{equation} \label{eq:exp_meas_id}
\int_\R f\,d\E\mu = \E\biggl[\int_\R f\,d\mu\biggr].
\end{equation}
\item
If $\int_\R |f| \,\E\mu < \infty$, then $\int_\R f\,d\mu$
is a.s.\ finite and measurable, and we have \eqref{eq:exp_meas_id}.
\end{enumerate}
\end{lemma}

\begin{proof}
Since (2) follows immediately from (1), we will prove (1) only.
As the statement of (1) holds for $f=1_{(-\infty,x]}$ for all $x \in \R$,
Dynkin's $\pi$-$\lambda$ theorem implies that the statement holds for
all measurable $A \subset \R$.
By the simple function approximation argument, the statement extends
to all nonnegative measurable $f$.
\end{proof}

In order to talk about $\E\mu_{W_n}$, we first need to establish
the measurability of $\mu_{W_n}$ for each $x \in \R$.

\begin{lemma}[measurability]
If $W$ is a random $n\times n$ Hermitian matrix, then $\mu_W$ is measurable.
\end{lemma}

\begin{proof}
Let $f(x) := \det(xI-W)$.
Given an interval $[a,b]$ where $a < b$, the event of having an eigenvalue
of $W$ with multiplicity at least $k$ (where $k\le n$) in $[a,b]$
is equal to
\[
\Bigl\{ \inf_{q \in [a,b] \cap \Q}
\bigl((f(q))^2+(f'(q))^2+\cdots+(f^{(k-1)}(q))^2\bigr) = 0 \Bigr\},
\]
which is indeed measurable.
Using this and by partitioning $\R$ into many small intervals,
one can show that $\{\lambda_j(W) \le x\}$ is measurable
for each $j=1,\ldots,n$ and $x \in \R$.
Since
\[ \mu_W(-\infty,x] = \frac{1}{n}\sum_{j=1}^n \I\{\lambda_j(W) \le x\}, \]
$\mu_W$ is measurable.
\end{proof}

Now we turn to the proof of Lemma \ref{lem:as_reduct}.
We need the following inequality,
which was found independently by Guntuboyina and Leeb \cite{GL09},
and Bordenave, Caputo, and Chafa\"i \cite{BCC11}.

\begin{lemma}[concentration for spectral measures] \label{lem:concen_spec}
Let $(W_n)_{n\in\N}$ be a Hermitian Wigner ensemble.
If the total variation of $f \colon \R \to \R$ is less than or equal to $1$,
then
\[
\P\biggl( \Bigl| \int_\R f\,d\mu_{W_n} - \E\int_\R f\,d\mu_{W_n} \Bigr| \ge t
\biggr) \le 2\exp(-nt^2/2).
\]
\end{lemma}

\begin{proof}
See \cite[Lemma C.2]{BCC11}.
\end{proof}

\begin{proof}[Proof of Lemma \ref{lem:as_reduct}]
Assume $\E \mu_{W_n} \tto \musc$.
For each $p,q \in \Q$ with $p < q$, let $f_{p,q} \colon \R \to \R$
be $1$ on $(-\infty,p]$, $0$ on $[q,\infty)$, and linear and continuous
on $[p,q]$.
Then
\[ \E\biggl[\int_\R f_{p,q} \,d\mu_{W_n}\biggr] \to
\int_\R f_{p,q} \,d\musc. \qquad \text{for all rational $p<q$.} \]
By Lemma \ref{lem:concen_spec} and the Borel-Cantelli lemma,
\[ \int_\R f_{p,q} \,d\mu_{W_n} \to \int_\R f_{p,q} \,d\musc
\qquad \text{for all rational $p<q$, a.s.} \]
This proves $\mu_{W_n} \tto \musc$ a.s.

To show the converse, assume that $\mu_{W_n} \tto \musc$ a.s.,
and let $f \colon \R \to \R$ be continuous and bounded.
Since $f$ is bounded, we can apply the dominated convergence theorem
to obtain
\[
\E\biggl[\int_\R f\,d\mu_{W_n}\biggr] \to \int_\R f\,d\musc.
\]
This shows $\E\mu_{W_n} \tto \musc$.
\end{proof}

\section{Reductions} \label{app:reduct}

In this section, we prove that Theorem~\ref{thm:wigner_var} follows from Theorem~\ref{thm:wigner} (Lemma~\ref{lem:var_reduct}),
and that Theorem~\ref{thm:wigner} follows from Theorem~\ref{thm:red_wigner}
(Lemma~\ref{lem:reduct}).

\begin{lemma} \label{lem:var_reduct}
Theorem \ref{thm:wigner_var} follows from Theorem \ref{thm:wigner}.
\end{lemma}

To prove this, we use the following two lemmas.

\begin{lemma}[perturbation inequality] \label{lem:pert_ineq}
If $A$ and $B$ are $n \times n$ Hermitian matrices, and
$F_A$ and $F_B$ are the distribution functions of $\mu_A$ and $\mu_B$, then
\[ \bigl(L(F_A,F_B)\bigr)^3 \le \tr\bigl((A-B)^2\bigr) \]
where $L$ is the L\'evy metric.
\end{lemma}

\begin{proof}
See \cite[Theorem A.41]{BS10}.
\end{proof}

\begin{lemma} \label{lem:under_lind}
If \eqref{eq:lind} holds, then the following are true:
\begin{enumerate}
\item
\eqref{eq:margin} holds if and only if \eqref{eq:weak_margin} holds.
\item
\eqref{eq:row_one} holds if and only if \eqref{eq:weak_row_one} holds.
\end{enumerate}
\end{lemma}

\begin{proof}
From \eqref{eq:lind} it follows that
\[ \frac{1}{n} \sum_{i,j=1}^n \E[|\w|^2;|\w|>1] \to 0. \]

(1) For any $J_n\subset\{1,\ldots,n\}$ with $|J_n|/n \to 0$, we have
\begin{multline}
\biggl| \frac{1}{n}\sum_{i\in J_n}\sum_{j=1}^n \E|w_{ij}|^2 -
\frac{1}{n}\sum_{i\in J_n}\sum_{j=1}^n \E[|w_{ij}|^2;|w_{ij}|\le 1] \biggr|\\
\le \frac{1}{n} \sum_{i,j=1}^n \E[|\w|^2;|\w|>1] \to 0.
\end{multline}
Thus, 
\eqref{eq:margin} holds if and only if \eqref{eq:weak_margin} holds.

(2) Since
\begin{multline}
\biggl|
\frac{1}{n} \sum_{i=1}^n \Bigl| \sum_{j=1}^n \E|\w|^2 - 1 \Bigr|
- \frac{1}{n} \sum_{i=1}^n \Bigl|\sum_{j=1}^n \E[|\w|^2;|\w|\le1]-1 \Bigr|
\biggr| \\
\le \frac{1}{n} \sum_{i,j=1}^n \E[|\w|^2;|\w|>1] \to 0,
\end{multline}
\eqref{eq:row_one} holds if and only if \eqref{eq:weak_row_one} holds.
\end{proof}

\begin{proof}[Proof of Lemma \ref{lem:var_reduct}]
We first show the sufficiency direction of Theorem \ref{thm:wigner_var}.
Recall that we proved that \eqref{eq:weak_zero} and \eqref{eq:weak_lind}
follows from \eqref{eq:lind} and $\E w_{ij}=0$.
By Lemma~\ref{lem:under_lind}, we have \eqref{eq:weak_margin}
and \eqref{eq:weak_row_one}.
Thus, $\E\mu_{W_n} \tto \musc$ follows from Theorem~\ref{thm:wigner}.

Now let us show the necessity.
On top of \eqref{eq:weak_zero} and \eqref{eq:weak_lind}, we have
\eqref{eq:weak_margin} by Lemma~\ref{lem:under_lind}.
Thus, \eqref{eq:weak_row_one} follows from Theorem~\ref{thm:wigner},
and therefore we have \eqref{eq:row_one} by Lemma~\ref{lem:under_lind}.
\end{proof}

Next we reduce Theorem~\ref{thm:wigner} to Theorem~\ref{thm:red_wigner}.

\begin{lemma} \label{lem:reduct}
\begin{enumerate}
\item
The sufficiency part of Theorem~\ref{thm:red_wigner}
implies the sufficiency part of Theorem~\ref{thm:wigner}.
\item
The necessity part of Theorem~\ref{thm:red_wigner}
implies the necessity part of Theorem~\ref{thm:wigner}.
\end{enumerate}
\end{lemma}

We need the following lemma.

\begin{lemma}[reduction to vanishing bounds] \label{lem:neg_reduct}
Let $(W_n)_{n\in\N}$ be a Hermitian Wigner ensemble satisfying
\eqref{eq:weak_zero} and \eqref{eq:weak_lind}.
Then there exist $1/2 \ge \eta_1 \ge \eta_2 \ge \cdots$ with $\eta_n \to 0$
such that if we let
\[ W_n' \equiv (\w')_{i,j=1}^n := \bigl(\w \I\{|\w|\le \eta_n\}
- \E[\w;|\w|\le\eta_n] \bigr)_{i,j=1}^n, \]
then the following are true:
\begin{enumerate}
\item
$\E\mu_{W_n} \tto \musc$ if and only if $\E\mu_{W_n'} \tto \musc$.
\item
\eqref{eq:weak_row_one} if and only if
\begin{equation} \label{eq:weak_row_one_prime}
\frac{1}{n} \sum_{i=1}^n \Bigl| \sum_{j=1}^n \E|\w'|^2 - 1 \Bigr| \to 0.
\end{equation}
\item
\eqref{eq:weak_margin} if and only if
\begin{equation} \label{eq:weak_margin_prime}
\frac{1}{n} \sum_{i\in J_n} \sum_{j=1}^n \E|\w'|^2 \to 0
\quad \text{for any $J_n\subset\{1,\ldots,n\}$ with $|J_n|/n\to0$.}
\end{equation}
\end{enumerate}
\end{lemma}

The proof of (1) will be based on the following lemma.

\begin{lemma}[Bernstein's inequality] \label{lem:bernstein_ineq}
Suppose that $X_1,\ldots,X_n$ are independent real-valued random variables
with $|X_i| \le 1$ and $\E X_i = 0$ for $i=1,\ldots,n$.
If $S := X_1 + \cdots +X_n$, then
\[ \P(S \ge x) \le \exp\biggl(\frac{-x^2}{2(\E S^2 + x)}\biggr)
\qquad \text{for all $x>0$.} \]
\end{lemma}

\begin{proof}
The proof of \cite[M20]{Bil99} with a slight modification works.
\end{proof}

\begin{proof}[Proof of Lemma \ref{lem:neg_reduct}]
Choose $1/2 \ge \eta_1 \ge \eta_2 \cdots$
with $\eta_n \to 0$ such that
\begin{equation} \label{eq:weak_lind_vanish}
\frac{1}{n} \sum_{i,j=1}^n \P(|\w|>\eta_n) \to 0.
\end{equation}

\emph{Proof of (1).}
By Lemma \ref{lem:as_reduct}, it is enough to show that
$\mu_{W_n}\tto\musc$ a.s.\ if and only if $\mu_{W_n'}\tto\musc$ a.s.
Let $\widetilde{W} := (\w\I\{|\w|\le\eta_n\})_{i,j=1}^n$.
We will first show that $\mu_{W_n} \tto \musc$ a.s.\ if and only if
$\mu_{\widetilde{W}_n} \tto \musc$ a.s.
Note that
\[ \frac{\rank(W_n-\widetilde{W}_n)}{n}
\le \frac{2}{n} \sum_{1\le i\le j\le n} \I\{|\w|>\eta_n\}. \]
By Lemma \ref{lem:rank_ineq}, it is enough to show that the right side
tends to $0$ a.s.

Let $\eps > 0$ be given. By \eqref{eq:weak_lind_vanish}, we have some
$n_0 \in \N$ such that
\[ \sum_{1\le i\le j\le n} \P(|\w| > \eta_n) \le \eps n/2
\qquad \text{for all $n \ge n_0$.} \]
Since $\I\{|\w| > \eta_n\}$, $1 \le i \le j \le n$, are independent,
Bernstein's inequality (Lemma \ref{lem:bernstein_ineq}) implies
\[
\begin{split}
\P\biggl( \sum_{1\le i\le j\le n} &\I\{|\w|>\eta_n\} \ge \eps n \biggr) \\
&\le \P\biggl( \sum_{1\le i\le j\le n} \bigl(\I\{|\w|>\eta_n\}
- \P(|\w|>\eta_n)\bigr) \ge \eps n/2 \biggr) \\
&\le \exp\biggl(\frac{-\eps^2 n^2/8}{\sum_{1\le i\le j\le n}
\P(|\w|>\eta_n) + \eps n/2}\biggr) \\
&\le \exp\biggl(\frac{-\eps^2 n^2/8}{\eps n}\biggr) = \exp(-\eps n/8).
\end{split}
\]

As
\[
\sum_{n=1}^\infty \exp(-\eps n/8) < \infty \qquad \text{for all $\eps > 0$,}
\]
the Borel-Cantelli lemma implies
\[
\frac{1}{n} \sum_{1\le i\le j\le n} \I\{|\w|>\eta_n\} \to 0 \qquad \text{a.s.}
\]
This implies that $\mu_{W_n} \tto \musc$ a.s.\ if and only if
$\mu_{\widetilde{W}_n} \tto \musc$ a.s.\ as explained above.

To show that $\mu_{\widetilde{W}_n} \tto \musc$ a.s.\ if and only if
$\mu_{W_n'} \tto \musc$ a.s., use Lemma~\ref{lem:pert_ineq} to note that
\[
\bigl(L(\mu_{\widetilde{W}_n},\mu_{W_n'})\bigr)^3
\le \frac{1}{n}\sum_{i,j=1}^n \bigl(\E[\w;|\w|\le\eta_n]\bigr)^2.
\]
Since $|a^2-b^2| = |a-b||a+b|$, the difference between the right side
and
\[ \frac{1}{n} \sum_{i,j=1}^n \bigl(\E[\w;|\w|\le1]\bigr)^2 \]
is bounded above by
\[
\begin{split}
\frac{1}{n} \sum_{i,j=1}^n \bigl|\E[\w;\eta_n<|\w|\le1]\bigr| &\cdot
\bigl|\E[\w;|\w|\le1]+\E[\w;|\w|\le\eta_n]\bigr| \\
&\le \frac{2}{n}\sum_{i,j=1}^n \P(|\w|>\eta_n) \to 0.
\end{split}
\]
Thus, by \eqref{eq:weak_zero}, we have $L(\mu_{\widetilde{W}_n},\mu_{W_n'})
\to 0$ a.s.

\emph{Proof of (2).}
Since $\bigl||a|-|b|\bigr| \le |a-b|$, we have
\[
\begin{split}
\biggl|
\frac{1}{n} \sum_{i=1}^n &\Bigl|\sum_{j=1}^n \E[|\w|^2;|\w|\le1]-1 \Bigr| -
\frac{1}{n} \sum_{i=1}^n \Bigl| \sum_{j=1}^n \E|\w'|^2 - 1 \Bigr|
\biggr| \\
&\le \frac{1}{n}\sum_{i,j=1}^n \Bigl|
\E[|\w|^2;|\w|\le1] - \E|w'_{ij}|^2 \Bigr|  \\
&= \frac{1}{n} \sum_{i,j=1}^n \E[|\w|^2;\eta_n < |\w| \le 1]
+ \frac{1}{n} \sum_{i,j=1}^n \bigl(\E[\w;|\w|\le\eta_n]\bigr)^2.
\end{split}
\]
The first term on the right side is bounded above by
\[ \frac{1}{n} \sum_{i,j=1}^n \P(|\w| > \eta_n) \to 0, \]
and we have shown just above that the second term also tends to $0$.

\emph{Proof of (3).}
The difference between the left sides of \eqref{eq:weak_margin} and
\eqref{eq:weak_margin_prime} is also bounded above by
\[
\frac{1}{n} \sum_{i,j=1}^n \E[|\w|^2;\eta_n < |\w| \le 1]
+ \frac{1}{n} \sum_{i,j=1}^n \bigl(\E[\w;|\w|\le\eta_n]\bigr)^2,
\]
which tends to $0$ as $n\to\infty$.
\end{proof}

\begin{proof}[Proof of Lemma~\ref{lem:reduct}]
Assume that $W_n$ is given as in Theorem~\ref{thm:wigner},
and define $W'_n$ as in Lemma~\ref{lem:neg_reduct}.
If we let $\eps_n := 2\eta_n$, then $(W_n')_{n\in\N}$ satisfies
the conditions of Theorem~\ref{thm:red_wigner}.
In particular, \eqref{eq:margin} for $W_n'$ follows from
(3) of Lemma~\ref{lem:neg_reduct}.

\emph{Proof of (1).}
Assume that the sufficiency part of Theorem~\ref{thm:red_wigner} holds.
If \eqref{eq:weak_row_one} holds, then \eqref{eq:weak_row_one_prime} holds
by (2) of Lemma~\ref{lem:neg_reduct}, and thus $\E\mu_{W_n'}\tto\musc$
by the sufficiency part of Theorem~\ref{thm:red_wigner}.
Then (1) of Lemma~\ref{lem:neg_reduct} tells us that $\E\mu_{W_n}\tto\musc$.

\emph{Proof of (2).}
Assume that the necessity part of Theorem~\ref{thm:red_wigner} holds.
If $\E\mu_{W_n} \tto \musc$, then $\E\mu_{W_n'}\tto\musc$ by
(1) of Lemma~\ref{lem:neg_reduct}, and thus \eqref{eq:weak_row_one_prime}
holds by Theorem~\ref{thm:red_wigner}.
This in turn implies \eqref{eq:weak_row_one} by (2) of
Lemma~\ref{lem:neg_reduct}.
\end{proof}

% References
\bibliographystyle{alpha}
\bibliography{references}
\end{document}